\documentclass[12pt]{amsart}
\usepackage{color}
\usepackage[all]{xy}
\usepackage{amssymb}
\usepackage{tikz}
\usepackage{comment}
\usepackage[mathscr]{euscript}
\usepackage{cite}
\usepackage{enumerate}
\usetikzlibrary{arrows}

\calclayout

\date{June 30, 2023 (v3)}
\newtheorem{dummy}{anything}[section] 
\newtheorem{theorem}[dummy]{Theorem}
\newtheorem*{thma}{Theorem A}
\newtheorem*{thmb}{Theorem B}

\newtheorem*{thmd}{Theorem 11.2}
\newtheorem*{nonum}{Theorem}
\newtheorem{lemma}[dummy]{Lemma} 
\newtheorem{proposition}[dummy]{Proposition} 
\newtheorem{corollary}[dummy]{Corollary}
 
\theoremstyle{definition}
\newtheorem{definition}[dummy]{Definition}
 \newtheorem{example}[dummy]{Example}
 
 \newtheorem{remark}[dummy]{Remark}

 \newtheorem*{question}{Question}
 \newtheorem*{acknowledgement}{Acknowledgement}

\newcommand
{\eqncount}{\setcounter{equation}{\value{dummy}}%
\addtocounter{dummy}{1}}

\newcommand{\cH}{\mathcal H}

\newcommand{\cS}{\mathcal S}

\newcommand{\cE}{\mathcal E}

\newcommand{\bZ}{\mathbb Z}

\newcommand{\CP}{\mathbb C P}

\newcommand{\fQ}{Q}

\newcommand{\bbL}{\mathbb L}

\newcommand{\bbZ}{\mathbb Z}

\newcommand{\sI}{\mathscr I}

\newcommand{\cy}[1]{\bbZ/{#1}}
\newcommand{\wX}{\widetilde X}
\newcommand{\wM}{\widetilde M}
\newcommand{\wP}{\widetilde P}

\newcommand{\wB}{\widetilde B}
\newcommand{\bd}{\partial}
\newcommand{\vv}{\, | \,}
\newcommand{\trf}{tr}
\newcommand{\La}{\Lambda}
\newcommand{\Zpi}{\bbZ [\pi]}

\newcommand{\mmatrix}[4]{\left (\vcenter
{\xymatrix@C-2pc@R-2pc{#1&#2\\#3&#4} }
\right )}

\newcommand{\la}{\langle}
\newcommand{\ra}{\rangle}
\DeclareMathOperator{\gdim}{g-dim}
\DeclareMathOperator{\Hom}{Hom}
\DeclareMathOperator{\wh}{Wh}

\DeclareMathOperator{\Sharp}{\#}
\DeclareMathOperator{\Image}{Im}

\DeclareMathOperator{\Ext}{Ext}

\DeclareMathOperator{\Tor}{Tor}

\DeclareMathOperator{\Isom}{Isom}

\DeclareMathOperator{\ad}{ad}

\newcommand{\RAAG}{right-angled Artin group} 

\newcommand{\RAAGs}{right-angled Artin groups}

\DeclareMathOperator{\cd}{cd}

\newcommand{\spinp}{spin$^{+}$}
\DeclareMathOperator{\id}{id}
\DeclareMathOperator{\sr}{{\mathfrak sr}}

\newcommand{\quadtypeMb}{[\pi_1(M,x_0), \pi_2(M), k_M, s_M]}
\newcommand{\quadtypeNb}{[\pi_1(N,x_0), \pi_2(N), k_N, s_N]}

\newcommand{\Ospin}{\Omega^{Spin}}
\newcommand{\rOspin}{\widehat\Omega^{Spin}}
\newcommand{\htildeM}{\widetilde \cH (M)}

\newcommand{\htildeMt}{\widetilde \cH (M_t)}
\newcommand{\hept}[1]{\hepta_{\bullet}(#1)}
\DeclareMathOperator{\hepta}{Aut}
\newcommand{\heqpt}[1]{\cE_{\bullet} (#1)}
\newcommand{\he}[1]{\hepta(#1)}
\newcommand{\hM}{\cH (M)}

\newcommand{\hMt}{\cH (M_t)}

\newcommand{\twotypeM}{[\pi_1(M), \pi_2(M), k_M]}

\newcommand{\LF}{{\scriptscriptstyle LF}}
\newcommand{\hL}{\widehat \Lambda}
\newcommand{\hN}{\widehat \Sigma}
\newcommand{\hsx}{\hphantom{xx}}

\begin{document}

\title[A stability range for $4$-manifolds]
{A stability range for topological $4$-manifolds}
\author{Ian Hambleton} 
\address{Department of Mathematics \& Statistics
 \newline\indent
McMaster University
 \newline\indent
Hamilton, ON  L8S 4K1, Canada}
\email{hambleton{@}mcmaster.ca}
\thanks{Research partially supported by NSERC Discovery Grant A4000. The author would also like to thank the Max Planck Institut f\"ur Mathematik in Bonn for its hospitality and support in June, 2022.}

\begin{abstract}\noindent
We introduce a new stable range  invariant for the classification of closed, oriented topological  $4$-manifolds  (up to $s$-cobordism), after stabilization by connected sum with a uniformly bounded number of copies of $S^2\times S^2$. 
\end{abstract}
\maketitle

\section{Introduction}
Due to recent work on the stable classification of topological $4$-manifolds, the outline of a general theory is emerging (see 
\cite{Kasprowski:2017}, \cite{Kasprowski:2021a}, \cite{Kasprowski:2021},  \cite{Kasprowski:2022},  \cite{Kasprowski:2021b}).
The most effective approach so far is a development of the original results of 
Wall \cite[Theorem 3]{wall-4man1}, \cite[Theorem 1]{wall-4man2}: if $M$ and $N$ are closed, simply connected, smooth $4$-manifolds with isomorphic intersection forms, then $M\# r(S^2 \times S^2)$ is diffeomorphic to $N\# r(S^2 \times S^2)$, for some $r \geq 0$. If this conclusion holds, we say that $M$ and $N$ are \emph{stably diffeomorphic}. The analogous notion for topological $4$-manifolds is \emph{stable homeomorphism}.

The following  result of Kreck \cite{Kreck99} provides a fruitful starting point for studying the stable classification problem in general:

\begin{nonum}[Kreck {\cite[Theorem 2]{Kreck99}}] Suppose that $M$ and $N$ are closed, smooth, spin $4$-manifolds, with the same fundamental group $\pi$ and equal Euler characteristics. If   $M$ and $N$ are spin bordant over $K(\pi,1)$, then $M\# r(S^2 \times S^2)$ is diffeomorphic to $N\# r(S^2 \times S^2)$, for some $r\geq 0$.
\end{nonum}
In this note, we consider the \emph{computability} of the number of stabilizations. 

\begin{question} If  $M\# r(S^2 \times S^2)$ is homeomophic to $N\# r(S^2 \times S^2)$, can one determine the minimum value of $r$  needed ? Is there a uniform estimate for the number of stabilizations, depending only on the fundamental group as $M$ and $N$ vary ?
\end{question}
The case of simply connected smooth $4$-manifolds is still not completely settled: no examples are known which require at least two copies of $S^2 \times S^2$ (instead of one copy) to achieve stable diffeomorphism. For stable homeomorphism of topological $4$-manifolds with finite fundamental group, one copy of 
 $S^2 \times S^2$ will suffice (see \cite[Theorem B]{hk5}).

\begin{remark}
One obstacle to determining optimal stabilization bounds is the failure of the $5$-dimensional  $s$-cobordism theorem for smooth manifolds, and its unknown status (in general) for  topological manifolds. To avoid this problem, we will aim for stability bounds for $s$-cobordisms rather than for homeomorphisms or diffeomorphisms.
\end{remark}

Our main result uses a new stable range (integer) invariant $\sr(\pi)$, depending only on a given finitely presented   group $\pi$ (see Definition \ref{def:threeone}). We will assume the assembly map properties (W-AA)  for $\pi$ given in Definition \ref{def:twoone}, and restrict attention to groups of \emph{type F}, meaning that there exists a \emph{finite} aspherical $n$-complex with fundamental group $\pi$, for some $n\geq 0$. In this case, we say that   $\pi$ is  \emph{geometrically $n$-dimensional} ($\gdim(\pi) \leq n$).

\begin{thma} Let $\pi$ be an discrete group  of type $F$ satisfying properties \textup{(W-AA)}. Let
$M$ and $N$ be closed, smooth, spin $4$-manifolds  with  fundamental group $\pi$,  which are oriented homotopy equivalent. 
 Then $M\# r(S^2 \times S^2)$ is smoothly $s$-cobordant to $N\# r(S^2 \times S^2)$,
  provided that $r \geq \sr(\pi)$. 
\end{thma}

\begin{remark} A similar result holds for topological $4$-manifolds.
 If $M$ and $N$ are smooth and simply-connected, Theorem A (with $r=0$) was proved by Wall \cite[Theorem 2]{wall-4man1}. In this case, the homotopy type  is determined by the intersection form.
 \end{remark}
\begin{example} For $\pi$ a \RAAG\ with $\gdim(\pi) \leq 4$ the assembly map conditions hold, and $\sr(\pi) \leq 6$ by  Proposition \ref{cor:raag}. 
\end{example}

For a non-simply connected $4$-manifold $M$, the basic homotopy invariants are the fundamental group $\pi :=\pi_1(M)$, the second homotopy group $\pi_2(M)$, the equivariant intersection form $s_M$ on 
$\pi_2(M)$, and the first $k$-invariant, $k_M \in H^3(\pi;\pi_2(M))$.  These invariants give the \emph{quadratic $2$-type}
$$ \fQ(M): = [\pi_1(M), \pi_2(M), k_M, s_M]$$
whose \emph{isometry type} largely determines the classification up to $s$-cobordism of closed oriented  topological 4-manifolds with geometrically $2$-dimensional fundamental groups (see \cite{HKT09} for the precise conditions). The appropriate notion of (oriented) isometry is given in Definition \ref{def:isometry}. 

\begin{question} How strong an invariant is the quadratic $2$-type ? Does $\fQ(M)$ determine the homotopy type of $M$ ? The stable homeomorphism type   (if $M$ is spin) ?
\end{question}

We will concentrate on geometrically finite fundamental groups, which in particular are torsion-free (see \cite[Proposition 9.2]{Kasprowski:2022} for an example with $\pi = \bZ \times \cy p$, showing that $Q(M)$ does not determine the homotopy type for $M = L^3(p,q) \times S^1$).

Here is a sample application of the stable range invariant for manifolds $M$ with a given quadratic $2$-type. In the statement, $d(\pi)$ denotes the minimal number of generators for $\pi$ a finitely generated group.

\begin{thmb} Let $\pi$ be the fundamental group of a closed, oriented, aspherical  $3$-manifold. Suppose that $M$ and $N$ are closed, topological, spin $4$-manifolds  with  fundamental group $\pi$,  and  isometric oriented  quadratic $2$-types.  
 If $M$ and $N$ are stably homeomorphic,  then $M\# r(S^2 \times S^2)$ is $s$-cobordant to $N\# r(S^2 \times S^2)$,
  provided that $r \geq 2d(\pi)$. 
\end{thmb}

For manifolds with fundamental groups in this class, the stable classification was completely carried out by Kasprowski, Land, Powell and Teichner \cite{Kasprowski:2017} (compare \cite[Theorem B]{Hambleton:2019}).
We remark that stable range invariants for noetherian rings due to Bass \cite{bass.h.1973.1}, Stafford \cite{stafford1} and Vaserstein \cite{vaserstein1} have previously been used to obtain bounds on the number of stabilizations required (for example, see  \cite[Theorem B]{hk5} for finite fundamental group, and \cite[Theorem 1.1]{Crowley:2011}, \cite[Theorem 2.1]{Khan:2017}). It is not clear at present how these more ``arithmetic" stability bounds are related to the $L$-theory bound used here. Another kind of ``non-cancellation" result arises from relating invariants of finite $2$-complexes to the stabilization of their $4$-manifold thickened doubles (see \cite{kreck-schafer1}).

\tableofcontents

\begin{acknowledgement} I am indebted to Daniel Kasprowski for pointing out a key issue in carrying out the last step of modified surgery, and to Diarmuid Crowley, Peter Kropholler, Wolfgang L\"uck and Larry Taylor for helpful conversations. I am very grateful for all the comments and suggestions from the referees. Thanks to Danny Ruberman for providing Ronnie Lee's notes ($\sim$1970) on the stable range for  $L_5(\bZ[\bZ])$. I would also like to express my appreciation to Matthias Kreck and Peter Teichner for our many discussions and collaborations over the years. \end{acknowledgement}

\section{The  quadratic $2$-type}\label{sec:one}
Here is a brief summary of the definitions in \cite{Hambleton:1988a} and \cite{HKT09}. 

\begin{definition}\label{def:equiv}
For an oriented 4-manifold $M$, the \emph{equivariant intersection form} is the triple $(\pi_1(M,x_0), \pi_2(M,x_0), s_M)$, where $x_0\in M$ is a base point and 
\begin{equation*}
s_M\colon  \pi_2(M,x_0) \otimes_\bbZ \pi_2(M,x_0) \to \La,
\end{equation*}
where $\La :=\bbZ[\pi_1(M,x_0)]$. This pairing is derived from the cup product on $H^2_c(\wM;\bbZ)$, where $\wM$ is the universal cover of $M$;  we identify $H^2_c(\wM;\bbZ)$ with $\pi_2(M)$ via Poincar\'e duality and the Hurewicz Theorem, and so $s_M$ is defined by
\begin{equation*}
s_M(x,y) = \sum_{g \in \pi} \varepsilon_0(\tilde x \cup \tilde yg^{-1}) \cdot g \in \Zpi ,
\end{equation*}
where $\tilde x, \tilde y \in H^2_c(\wM;\bbZ)$ are the images of $x,y \in \pi_2(M)$ under the composite isomorphism $\pi_2(M) \to H_2(\wM;\bbZ) \to H^2_c(\wM;\bbZ)$ and $\varepsilon_0$ is given by $\varepsilon_0\colon H^4_c(\wM;\bbZ) \to H_0(\wM;\bbZ)=\bbZ$.

Alternately, we can identify $H^2_c(\wM;\bbZ) = H^2(M;\La)$, and define $s_M$ via cup product and evaluation on the image $\trf[M] \in H^\LF_4(M;\bZ)$ of the fundamental class of $M$ under transfer.
\end{definition}
Unless otherwise mentioned, we work with pointed spaces and maps,  and our  modules are  \emph{right} $\La$-modules.
The pairing $s_M$ is $\Lambda$-hermitian, meaning that for all $\lambda\in\Lambda$, we have
\[ s_M( x, y\cdot \lambda)= s_M(x,y)\cdot \lambda  \quad \text{ and } \quad s_M(y, x)= \overline{s_M(x,y)},
\]
where $\lambda\mapsto \bar\lambda$ is the involution on $\Lambda$ given by the orientation character of $M$. This involution is determined by $\bar g = g^{-1}$ for $g \in \pi_1(M,x_0)$. For later reference, we note that when $M$ is spin the  term
$\varepsilon_0(\tilde x, \tilde y) \equiv 0 \pmod 2$, so $s_M$ is an \emph{even} hermitian form.

Let $B:= B(M)$ denote the  algebraic $2$-type of a closed oriented topological $4$-manifold  $M$ with infinite fundamental group $\pi$. In particular, the classifying map $c\colon M \to B$ is $3$-connected and $B$ is $3$-co-connected. The space $B$ is determined up to homotopy equivalence by the algebraic data $\twotypeM$.

\medskip\noindent
{\bf Notation}: In the rest of the paper, if the coefficients for homology groups are not explicitly stated, then  we mean \emph{integral homology}  $H_*(-; \bZ)$.

\medskip
We will assume that $\pi$ is infinite and one-ended, or equivalently that $H^1(\pi; \La) = 0$. By Poincar\'e duality, this implies that $H_3(\wM;\bZ) = H_3(M;\La) = 0$. 
Under these assumptions, $$H_4(M) \cong H_4(M, \wM) \cong H_4(B, \wB) \cong \bZ$$ (see the proof of Proposition \ref{prop:sixone}(i) for the details), and we let $\mu_M \in  H_4(B, \wB)$ denote the image $\mu_M = c_*[M]$ of the fundamental class of $M$ under this composite. We regard the class $\mu_M$ as an \emph{orientation} of the quadratic $2$-type.

\begin{definition}\label{def:enriched} The \emph{oriented quadratic $2$-type} is the 4-tuple:
$$\fQ (M):= \quadtypeMb$$
together with the class $\mu_M \in H_4(B, \wB)$.
\end{definition}

\begin{definition}\label{def:isometry} An \emph{orientation-preserving isometry} of   quadratic $2$-types $\fQ(M)$ and $\fQ(N)$ is a triple $(\alpha, \beta, \phi)$, such that
\begin{enumerate}
\item $\alpha\colon \pi_1(M,x_0) \to \pi_1(N,x'_0)$ is an isomorphism of fundamental groups;
\item $\beta\colon (\pi_2(M), s_M) \to (\pi_2(N), s_{N})$  is an $\alpha$-invariant isometry of the equivariant intersection forms, such that $(\alpha^*, \beta_*^{-1})(k_{N}) = k_{M}$;
\item  $\phi\colon B(M) \to B(N)$ is a base-point preserving homotopy equivalence lifting $(\alpha,\beta)$, such that $\phi_*(\mu_M) = \mu_N$.
\end{enumerate}

\smallskip
\noindent
  In addition, the following diagram
$$\vcenter{\xymatrix{0 \ar[r]& H^2(\pi; \La) \ar[r]\ar@{=}[d]& H^2(N; \La) \ar[r]^(0.4){e_{N}}\ar[d]_{\cong}^{\beta^*}&
\Hom_\La(H_2(N;\La),\La) \ar[r]\ar[d]_{\cong}^{\beta^*}&H^3(\pi;\La) \ar[r]\ar@{=}[d]& 0\cr
0 \ar[r]& H^2(\pi; \La) \ar[r]& H^2(M; \La) \ar[r]^(0.4){e_{M}}&
\Hom_\La(H_2(M;\La),\La) \ar[r]&H^3(\pi;\La) \ar[r]& 0}}
$$
arising from the universal coefficient spectral sequence \emph{commutes}, with maps $e_M$, $e_{N}$ induced by evaluation, and  $\beta$ after identifying $\pi:= \pi_1(M,x_0) \cong \pi_1(N,x'_0)$ via $\alpha$. We will  assume throughout that our manifolds are connected, so that a change of base points leads to isometric intersection forms.
By a \emph{stable} isometry, we mean an oriented isometry of   quadratic $2$-types after adding a hyperbolic form $H(\La^r)$ to both sides. 
\end{definition}

\begin{remark}
Recall that there is an exact sequence of groups
$$ 1 \to H^2(\pi; \pi_2(B)) \to \hept{B} \to \Isom(\twotypeM) \to 1$$
detemining the group $\hept{B}$ of base-point preserving homotopy self-equivalences of $B$ up to extension (see M\o ller \cite[\S 4]{Moeller:1991}). In Proposition \ref{prop:sixone}(iv) we will show that the image $\phi_*(\mu_M)$ depends only the isometry induced by $\phi$ on the algebraic $2$-type of $M$. In particular, this implies that the condition $(iii)$ above is independent of the choice of $\phi$.
\end{remark}

\section{Modified surgery and assembly maps}\label{sec:two}
A  standard approach to the classification of  topological $4$-manifolds uses the theory of ``modified surgery" due to Matthias Kreck  \cite[\S 6]{Kreck99}.
We briefly recall some of the features of modified surgery in our setting (see \cite[Theorem 4, p.~735]{Kreck99} for the notation): 
\begin{itemize}
\item Let $M$ and $N$ be closed, oriented topological $4$-manifolds with the same Euler characteristic, which admit normal 1-smoothings in a fibration $B \to BSTOP$.  
\item If $W$ is a normal $B$-bordism between these two 1-smoothings, with normal $B$-structure $\bar\nu$, then there exists an obstruction $\Theta(W, \bar\nu) \in \ell_5(\pi_1(B))$ which is \emph{elementary} if and only if $(W,\bar\nu)$ is $B$-bordant relative to the boundary to an $s$-cobordism.

\item Let $\pi:= \pi_1(B)$ and $\La:= \bZ [\pi]$ denote the integral group ring of the fundamental group.
The elements of $\ell_5(\pi)$ are represented by pairs $(H(\La^r), V)$, where $V$ is a half-rank direct summand of the hyperbolic  form $H(\La^r)$. 

\item In a pair $(H(\La^r), V)$, if  the quadratic
form vanishes on $V$, then the element $\Theta(W, \bar\nu)$  lies in the image of $L_5(\bZ \pi) \to \ell_5(\pi)$ (see \cite[Proposition 8, p.~739]{Kreck99} or \cite[p.~734]{Kreck99}  for criteria to ensure that this will happen).
\end{itemize}
In many applications of modified surgery, the last step involves using assembly maps in $K$-theory and $L$-theory to eliminate an obstruction in $L_5(\bZ \pi)$. We will give an overview of this technique, starting with a description of the relevant assembly maps.

\medskip
 Let $\bbL= \bbL(\bZ)$ denote the non-connective periodic $L$-spectrum of the integers, and let $\bbL_\bullet$ denote its $0$-connective cover (with the space $G/TOP$ in dimension zero). By construction, we have an identification $\Omega^4 \bbL(i) = \bbL(i)$, for $i \in \bZ$.
 The connective assembly maps (for $k\geq 0$)
 $$A^\bullet_{n + 4k}(\pi)\colon H_{n+4k}(\pi; \bbL_\bullet) \to  L^s_{n + 4k}(\Zpi)$$
 are related to maps in the geometric surgery exact sequence. The (non-connective) assembly maps (for $n \geq 5$) can be expressed  as the composite:
 $$A_n(\pi) \colon H_n(\pi; \bbL) = \varinjlim_{k} \, H_{n+4k}(\pi; \bbL_\bullet) \xrightarrow{\ \varinjlim A^\bullet_{n + 4k}(\pi)\ } 
\varinjlim_{k} \,   L^s_{n + 4k}(\Zpi) \xrightarrow{\hsx\hsx} L^s_n(\Zpi)$$
  where the maps to the direct limit are induced by suspension and the composition:
  $$ B\pi_{+}\wedge \Sigma^{4k}\bbL(i) \to
   B\pi_{+}\wedge \Sigma^{4k}\Omega^{4k}\bbL(i) \to B\pi_{+}\wedge \bbL(i).$$
  The  periodicity isomorphisms $L^s_n(\Zpi) \cong L^s_{n + 4k}(\Zpi)$  are defined geometrically by ``crossing" with $\CP^2$
 (see \cite[\S\S 10-11]{Hambleton:2004a}  and \cite[\S 7.3]{Luck:2020} for the connection between assembly maps and surgery). 
 
 We will be interested in the assembly maps $A_5(\pi) $ and $A^\bullet_4(\pi)$. Note that  
 $$A^\bullet_5(\pi)\colon H_5(\pi;\bbL_\bullet) \cong H_1(\pi; \bZ) \oplus H_3(\pi;\cy 2)
  \xrightarrow{\ \sI_1 \oplus \kappa_3\ } L^s_5(\Zpi),$$
  but  in general the groups $H_5(\pi; \bbL)$ involve the higher homology of $\pi$ (localized at 2) and $KO_*(\pi)$ (localized at odd primes), as explained in
 \cite[Theorem A]{Taylor:1979}. The two components of $A^\bullet_5(\pi)$ are given by ``universal homomorphisms" $\sI_1(\pi)\colon H_1(\pi;\bZ) \to L^s_5(\Zpi) $ and $\kappa_3(\pi)\colon H_3(\pi;\cy 2) \to L^s_5(\Zpi)$ (see \cite[\S 1C]{Hambleton:1988}).
  If $\gdim(\pi) \leq 4$ then $H_5(\pi;\bbL_\bullet) \cong H_5(\pi;\bbL)$ and   $A^\bullet_5(\pi) = A_5(\pi) $.

 To obtain concrete applications, it is convenient to assume the following conditions.

\begin{definition}\label{def:twoone}
 A group $\pi$ satisfies properties (W-A) whenever
\begin{enumerate}
\item The Whitehead group $\wh(\pi)$ vanishes. 
\item The assembly map $A_5\colon H_5(\pi;\bbL) \to L^s_5(\Zpi)$ is surjective. 
\end{enumerate}
If,  in addition,  the assembly map $A^\bullet_4(\pi)\colon H_4(\pi;\bbL_\bullet) \to L_4(\Zpi)$ is injective, we say that $\pi$ satisfies properties (W-AA). In particular, since $H_4(\pi;\bbL_\bullet) \cong H_0(\pi;\bZ) \oplus H_2(\pi; \cy 2)$ the condition (W-AA) implies that the second component $\kappa_2(\pi)\colon H_2(\pi;\cy 2) \to 
L^s_4(\Zpi)$ of the assembly map $A^\bullet_4(\pi)$  is injective (see \cite[\S 1]{Hambleton:1988}).
\end{definition}

\begin{remark} In the rest of the paper, we will usually be assuming that $\wh(\pi)=0$, and we will write $L_*(\Zpi)$ (undecorated) to mean any of the $L$-theories based on subgroups of the Whitehead group. The Farrell-Jones assembly map conjectures \cite{Farrell:1993} are usually expressed with target $L^{-\infty}(\Zpi)$, the $L$-theory with decorations based on the non-connective $K$-spectrum. 
For torsion-free groups, these conjectures imply results about the  assembly maps used in Definition \ref{def:twoone}           
(see \cite[Conjecture 1.19 and Corollary 2.11]{Luck:2005}).
\end{remark}

\begin{lemma}\label{lem:kappa} Let $\pi$ be a torsion-free discrete group which satisfies the Farrell-Jones isomorphism conjectures in $K$-theory and $L$-theory. Then the (connective) assembly map $A^\bullet_4(\pi)$ is injective. 
\end{lemma}
\begin{proof} Let $\bbL_{(2)}$ denote the $2$-localization of the periodic $L$-spectrum. If $\pi$ satisfies the Farrell-Jones isomorphism conjectures in $K$-theory and $L$-theory, then the (non-connective) assembly map
$$A_4(\pi)\colon H_4(\pi;\bbL) \to L^{-\infty}_4(\Zpi)$$
is an isomorphism. If $\pi$ is torsion-free, the isomorphism holds for $L^s$ by \cite[Corollary 2.11]{Luck:2005} and we can omit the decorations.  Hence the 2-localization
$$A_4(\pi) \colon H_4(\pi;\bbL)_{(2)} \to L_4(\Zpi)_{(2)}$$
 is also an isomorphism. Since the $L$-spectra localized at 2 are products of Eilenberg-MacLane spectra, the comparison map
 $$i_\bullet \colon H_4(\pi;\bbL_\bullet)_{(2)} \to H_4(\pi;\bbL)_{(2)} \cong H_4(\pi;\bbL_{(2)}) $$
 is an injection (both are products of certain 2-local homology groups of $\pi$). 
 
 We have a commutative diagram:
 $$\xymatrix{H_4(\pi;\bbL_\bullet) \ar[r]^{i_\bullet}\ar@{>->}[d] & H_4(\pi;\bbL)\ar[r]_{\approx}^{A_4}
\ar[d]  &  L_4(\Zpi)  \ar[d] \\
H_4(\pi;\bbL_\bullet)_{(2)}  \ar@{>->}[r]^{i_\bullet}\ & H_4(\pi;\bbL)_{(2)}  \ar[r]_{\approx}^{A_4}&  L_4(\Zpi) _{(2)}
}$$ 
Moreover, since $H_4(\pi;\bbL_\bullet)\cong H_0(\pi; \bZ) \oplus H_2(\pi;\cy 2)$, the 2-localization map
 $$H_4(\pi;\bbL_\bullet) \to H_4(\pi;\bbL_\bullet)_{(2)}$$
is injective, and hence the assembly map $A_4^{\bullet}(\pi)\colon H_4(\pi;\bbL_\bullet) \to L_4(\Zpi)$ is injective,
\end{proof}

\begin{remark}\label{rem:threefour}
We  conclude from Lemma \ref{lem:kappa} that the properties (W-AA) hold for the assembly maps into the surgery obstructions groups $L^s_*(\Zpi)$,  whenever the group $\pi$ is torsion-free   and satisfies the Farrell-Jones isomorphism conjectures in $K$-theory and $L$-theory (see \cite[Theorem 11.2(5)]{Luck:2020}).
These conjectures have been verified for many classes of groups, and in particular for all \RAAGs\ 
 (see \cite{Bartels:2012}, \cite{Bartels:2014}). \end{remark}

From surgery theory, we know that the action of elements in the image $\Image A^\bullet_5(M) \subseteq L_5(\bZ \pi)$ of the assembly map on  $\Theta(W, \bar\nu) \in \ell_5(\pi_1(B))$  can be defined geometrically by the action of  degree 1 normal maps on the $B$-bordism $(W, \bar\nu)$. Here 
$$A^\bullet_5(M) \colon H_5(M;\bbL_\bullet) = H_1(M;\bZ) \oplus H_3(M; \cy 2) \to L_5(\Zpi)$$
is defined by the surgery obstructions of degree 1 normal maps 
$$F\colon ( U, \bd_0 U, \bd_1 U) \to  (M \times I, M \times 0, M \times 1).$$
 By definition, $\bd_0 U = \bd_1 U = M$, and $F$ restricted to both boundary components is a homeomorphism. Such \emph{inertial} normal cobordisms can be glued to $(W, \bar\nu)$ to produce a new $B$-bordism $(W', \bar\nu)$ between $M$ and $N$, with surgery obstruction $\Theta(W', \bar\nu) =\Theta(W, \bar\nu) + \sigma(F)$ (see the proof of \cite[Theorem 2.6]{HKT09}). 
 
 \medskip
 This is the argument used in \cite[Theorem C]{HKT09} for the final step, where the fundamental groups $\pi$ were assumed geometrically $2$-dimensional,  to eliminate the obstruction $\Theta(W, \bar\nu)$,  and thus obtain an $s$-cobordism between $M$ and $N$.  We assumed that the assemby map  $A_5^\bullet(\pi)$ was surjective.
 
 In  \cite[Theorem 11.2]{Hambleton:2019}, the same argument was proposed to obtain a classification of closed, \spinp, topological  $4$-manifolds with fundamental group $\pi$ of cohomological dimension $\leq 3$ (up to $s$-cobordism), after stabilization by connected sum with at most $b_3(\pi)$ copies of $S^2\times S^2$. The goal of this work was to obtain an $s$-cobordism after a uniformly bounded number of stabliizations, where the bound depends only on the fundamental group.
 
 \medskip
 However, there was an error in this outline for  \cite[Theorem 11.2]{Hambleton:2019}  which is now addressed in Section \ref{sec:five} by using the new stable range invariant 
 (see \cite{Hambleton:2021}). We record the issue  which led to 
 the error (as a ``warning"), since it may arise in other applications of modified surgery.

\medskip
\noindent
{\bf Caveat}: The  domain of the (connective)  assembly map: 
$$A^\bullet_5(\pi) \colon H_5(\pi;\bbL_\bullet) = H_1(\pi;\bZ) \oplus H_3(\pi; \cy 2) \to L_5(\Zpi) $$
 is expressed in terms of the group homology of $\pi$.
 However,  the above construction can only realize the action of elements in the image of  the partial assembly map
$$ H_5(M;\bbL_\bullet) = H_1(M;\bZ) \oplus H_3(M; \cy 2)  \to H_5(\pi;\bbL_\bullet) \to L_5(\Zpi)$$
from the homology of $M$.
Since the reference map $M \to B$ is $2$-connected, the summand $H_1(M;\bZ) \cong H_1(\pi;\bZ)$. However, if the map $H_3(M;\cy 2) \to H_3(\pi;\cy 2)$ is not surjective, we will not be able to realize all possible obstructions by this construction.  

\begin{remark} The statements of \cite[Theorems 2.2 \& 2.6]{HKT09} are a bit misleading, since they appear (incorrectly) to be stated for arbitrary fundamental groups.  However, the goal of \cite{HKT09} was to study fundamental groups $\pi$ of geometric (and hence cohomological) dimension at most two. In these cases, $H_3(\pi;\cy 2) = 0$ so the domain of $A^{\bullet}_5(\pi)$ is just $H_1(\pi;\bZ)$, and  the problem above does not arise. In contrast, if $\cd \pi = 3$ and $\pi_1(M) = \pi$,  then by Poincar\'e duality:
$$\xymatrix{H^1(M;\cy 2) \ar[d]^{\cap [M]}_{\cong} & H^1(\pi;\cy 2) \ar[d]^{c_*{[M]} }\ar[l]^{\cong}\\
H_3(M;\cy 2) \ar[r] & H_3(\pi;\cy 2)}$$
and  the map $H_3(M;\cy 2) \to H_3(\pi; \cy 2)$ is zero since $0 = c_*{[M]} \in H_4(\pi; \cy 2)$.
\end{remark}

\section{A stable range for $L$-theory}\label{sec:three}
For any finitely presented group $\pi$, the odd dimensional surgery obstruction groups are defined as $L_5(\Zpi) = SU(\La)/RU(\La)$, in the notation of Wall \cite[Chap.~6]{wall-book}. Here $SU(\La)$ is the limit of the automorphism groups $SU_r(\La)$ of the hyperbolic (quadratic) form $H(\La^r)$ under certain injective maps
$$  \dots SU_r(\La) \to SU_{r+1}(\La) \to \dots \to SU(\La),$$
and $RU(\La)$ is a suitable subgroup determined by the surgery data, so that $L_5(\Zpi)$ is an abelian group. To define a stable range, we will assume that the fundamental groups are \emph{geometrically $n$-dimensional} ($\gdim(\pi) \leq n$), meaning that there exists a finite aspherical $n$-complex with fundamental group $\pi$.

We introduce a measure f the ``stability" of elements of $L_5(\Zpi)$ in the image of the assembly map. The first factor 
 of the comparison map 
$$i_\bullet \colon H_5(\pi;\bbL_\bullet) = H_1(\pi;\bZ) \oplus H_3(\pi; \cy 2) \to H_5(\pi;\bbL)$$
defines a subgroup
$$\sI_1(\pi) := \{ i_\bullet (u,0)) \vv u \in H_1(\pi; \bZ)\} \subset H_5(\pi;\bbL) $$

\begin{definition}\label{def:threeone}
 For an element $x \in L_5(\Zpi)$, we denote its \emph{stable $L_5$-range} by:
$$\sr (x) = \min \{r \geq 0:  x \textit{ \ is represented by a matrix in\ } SU_r(\La)\}.$$

The \emph{stable $L_5$-range of a group $\pi$ of type $F$} is defined as: 
$$\sr(\pi) = \min\nolimits_S\{\max\{\sr(A_5(\alpha)) :  \alpha\in     S\subset  H_5(\pi;\bbL)\}\}, $$ 
over all subsets $S$  which \emph{project to  generating sets} of the quotient $H_5(\pi;\bbL)/\sI_1(\pi)$.
\end{definition}
\begin{remark} In defining the stable range $\sr(\pi)$, we quotient out the subgroup $\sI_1(\pi)$, since   in our setting $H_1(M;\bZ) \cong H_1(\pi;\bZ)$ and stabilization is not needed to realize these obstructions. If $\pi= \bZ$ is infinite cyclic,  
Ronnie Lee\footnote{I am indebted to Danny Ruberman for providing Ronnie Lee's notes: available on request} (see \cite[Example 1.6]{Cappell:1971}) showed that $\sr(x) \leq 1$,  for all $x \in L_5(\Zpi)$.
\end{remark}

\begin{remark}\label{rem:threeone}
 If $\gdim(\pi) < \infty$,  then $H_5(\pi;\bbL)$ will be a finitely generated abelian group, and the stable range $\sr (\pi)$ will be finite.
 Without this assumption  $\sr(\pi)$ could be infinite, since there are finitely presented groups with $H_3(\pi;\cy 2)$ of infinite rank (see Stallings \cite{Stallings:1963}). The Stallings group $\pi$  is a possible example, since it has $\cd \pi = 3$ and satisfies the Farrell-Jones conjectures (see \cite{Bestvina:1997} and \cite[Theorem 1.1]{Bestvina:2021}). 
 \end{remark} 
 
 In the following statement, we let $d(\pi)$ denote the minimal number of generators for a finitely generated discrete group.
  
 \begin{lemma} Let $\pi$ denote the fundamental group of a closed, orientable $3$-manifold. Then $\sr(\pi) \leq 2 d(\pi)$.
 \end{lemma}

\begin{proof} Let $N^3$ be a closed, orientable $3$-manifold with fundamental group $\pi$.
By definition of the assembly map, we need to determine the minimum representative in $SU_r(\La)$ for the surgery obstruction of the degree one
normal map
$$ g:= (\id \times f)\colon N \times T^2 \to N \times S^2$$
given by the the product of the Arf invariant one normal map $f\colon T^2 \to S^2$ with the identity on $N$.  After surgery on the generators of 
$$K_1(g) = \ker\{H_1(N \times T^2; \La) \to H_1(N \times S^2; \La)\} = \bZ \oplus \bZ$$
we get a $2$-connected normal map with $K_2(g') = I(\rho) \oplus I(\rho)$, where $I(\rho) := \ker \{\bZ [\pi] \to \bZ\}$ is the augmentation ideal of 
the group ring $\bZ [\pi] $. According to the recipe provided by Wall \cite[Chap.~6, pp.~58-59]{wall-book}, the surgery obstruction is represented in 
$SU_r(\La)$, where $r \geq 2d(\pi)$ since an epimorphism $\La^r \to I(\rho)$ requires $r \geq d(\pi)$.
\end{proof}

\begin{corollary}\label{cor:raag} Let $\pi$ be a \RAAG\ with $\gdim(\pi) \leq 4$. Then $\sr(\pi) \leq 6$. 
\end{corollary}
\begin{proof}  Every \RAAG\ $\pi$ has $\gdim(\pi) < \infty$ since it is defined by a finite graph. As remarked above, 
$A_5^\bullet(\pi) = A_5(\pi)$ if $\gdim(\pi) \leq 4$. The homology group $H_3(\pi;\cy 2)$ has $\cy 2$-rank $b_3(\pi)$, which is equal to the number of $3$-cliques in the defining graph for $\pi$. Moreover, since each $3$-clique determines a subgroup $\bZ^3 \subseteq  \pi$,  the group $H_3(\pi;\cy 2)$ is generated by the images of the fundamental classes under all  the induced maps $H_3(T^3;\cy 2) \to H_3(\pi;\cy 2)$.  It is therefore enough to determine the stable range for $\rho = \bZ^3$.
\end{proof}

\begin{remark} If $\pi$ is a \RAAG\  with $\gdim(\pi) \leq n$, then a similar argument shows that $\sr(\pi) 
\leq\sr(\bZ^n)$ whenever $H_5(\pi;\bbL)$ is generated by the images of toral subgroups of $\pi$. Note that $\sr(\bZ^n) \leq n+3$ (see \cite[Theorem B]{stafford1}).
\end{remark}

We will use a stable range condition to realize the action of $L_5(\Zpi)$ on a $B$-bordism, after a suitable stabilization. The following statement is an application of   this result in the setting of Kreck \cite[Theorem 4]{Kreck99}.

\begin{proposition}\label{prop:threetwo}
 Let $\pi$ be a discrete group of type $F$ satisfying properties \textup{(W-A)}. Let
$M$ and $N$ be closed, oriented topological $4$-manifolds with the same Euler characteristic, which admit normal 1-smoothings in a fibration $B \to BSTOP$.  
Suppose that $(W, \bar\nu)$ is a normal $B$-bordism between these two 1-smoothings. If $r\geq \sr(\pi)$, then for any $x \in L_5(\Zpi)$ there exists a $B$-bordism $(W', \bar\nu)$ between the stabilized 1-smoothings $M' :=M \Sharp r(S^2 \times S^2)$ and $N':=N \Sharp r(S^2 \times S^2)$,  with $\Theta(W', \bar\nu)  = \Theta(W, \bar\nu)  + x \in \ell_5(\pi)$.
\end{proposition}

\begin{proof} By property (W-A), the assembly map $A_5(\pi)\colon  H_5(\pi; \bbL) \to L_5(\Zpi)$ is surjective. The elements $x \in L_5(\Zpi)$ in the image of $H_1(\pi;\bZ) \cong H_1(M;\bZ)$ are realized  without stabilization (see the discussion following Remark \ref{rem:threefour}). For the elements $x = A_5(\alpha) \in L_5(\Zpi)$ in the image of $\alpha \in H_5(\pi;\bbL)$, we use the stabilized version of Wall realization due to Cappell and Shaneson \cite[Theorem 3.1]{Cappell:1971}. 

Any element is the image of a finite sum $\alpha = \sum \alpha_i$ of elements of $H_5(\pi;\bbL)$, which each have stable $L$-range at most $\sr(\pi)$, after subtracting an element of $\sI_1(\pi)$ if necessary. Pick $r \geq \sr(\pi)$ and let $M':= M \Sharp r(S^2 \times S^2)$. The realization construction can be done (for each term $\alpha_i$ of the finite sum) in small disjoint intervals 
$$M' \times [t_{i-1}, t_i] \subset M' \times [0,1], $$
 with $0=t_0 < t_1 < \dots < t_k = 1$,  to produce  degree one normal maps 
$$F_i\colon ( U_i, \bd_0 U_i, \bd_1 U_i) \to  (M' \times [t_{i-1}, t_i] , M' \times t_{i-1}, M \times t_i), \quad 1\leq i \leq k,$$
such that $\bd_0 U_i = \bd_1 U_i = M' = M \Sharp r(S^2 \times S^2)$. The restrictions of $F_i$ to the boundary components have the property that $F_i\vv_{\bd_0 U} = \id $, and $F_i\vv_{\bd_1 U}:= f_i$  is a simple homotopy equivalence. In other words, this construction produces elements of the structure set $\cS(M')$ represented by self-equivalences of $M'$.

These  normal bordisms can be glued (at disjoint levels) into a collar $M'\times [0,1]$ attached to the stablization $W \natural\, r(S^2 \times S^2 \times I)$ of the given $B$-bordism, and the reference map to $B$ extended through $M$. After including all these bordisms, the induced homotopy equivalence with target $M' \times 1$ is the composite
$f := f_1\circ f_2 \circ \dots \circ f_k$.The surgery obstruction over the collar $M'\times [0,1]$ is  $x = A_5(\alpha) = \sum A_5(\alpha_i)$, and the result follows.
\end{proof}
The following application of the theory in Kreck \cite[\S 6]{Kreck99} may be useful in cases where a potentially harder bordism calculation is feasible.

\begin{corollary}\label{cor:threethree} If $M$ and $N$ are closed, oriented  or topological $4$-manifolds which admit   $B$-bordant normal 2-smoothings in the same fibration $B\to BSTOP$, then they are s-cobordant after at most $\sr(\pi)$ stabilizations, provided their fundamental group has type $F$ and satisfies properties~\textup{(W-A)}.
\end{corollary}
\begin{proof}
For normal $2$-smoothings of $M$ and $N$, the reference maps are $3$-connected. In this case,  Kreck \cite[p.~734]{Kreck99} shows that the surgery obstruction $ \Theta(W, \bar\nu)$ of a $B$-bordism $(W, \bar\nu)$ lies in the image of  $L_5(\Zpi) \to \ell_5(\pi)$. The result now follows from Proposition \ref{prop:threetwo}.
\end{proof}

\begin{remark} The results of Cappell and Shaneson \cite[Theorem 3.1]{Cappell:1971} and Kreck \cite[Theorem 4]{Kreck99} also apply in the smooth category, and we obtain the analogous smooth versions of Proposition \ref{prop:threetwo} and Corollary \ref{cor:threethree} for normal smoothings in fibrations $B \to BSO$.
\end{remark}
\begin{proof}[The proof of Theorem A] Let $M$ and $N$ are closed, smooth, spin $4$-manifolds  with  fundamental group $\pi$, and let $f\colon N \to M$ be an oriented homotopy equivalence. In the setting of modified surgery, we have normal $2$-smoothings $(N, f)$ and $(M, \id)$ into the same fibration $B \to BSO$, where $B = M \times BSPIN$.

Under our assembly conditions (W-AA), the homomorphism $\kappa_2\colon H_2(\pi; \cy 2) \to :L_4(\Zpi)$ is injective (see Lemma \ref{lem:kappa}).   It follows that the normal invariant 
$$\eta(f) \in [M, G/TOP] \cong H_4(M; \bbL_0) \cong H_2(M;\cy 2) \oplus \bZ$$
 has trivial surgery obstruction, and lies in the image 
$$\Image\{\pi_2(M)\otimes \cy 2 \to H_2(M; \cy 2)\} = \ker\{ H_2(M;\cy 2) \to H_2(\pi; \cy 2)\}$$
since the surgery obstruction is determined  by the ordinary signature  difference and $\kappa_2(\pi)$ (see \cite[\S 1]{Hambleton:1988}).
 By
\cite[Theorem 19]{Kirby:2001}, these normal invariants are all realized by homotopy self-equivalences  (pinch maps) of $M$. Hence we may assume that the normal invariant $\eta(f)$ is trivial. Therefore, there exists a normal cobordism
$$(F, b)\colon (W, \bd_0 W, \bd_1 W) \to (M\times I, M\times 0, M\times 1)$$
with  $F\vert_{\bd_0 W} = \id\colon M \to M$ and $F\vert_{\bd_1 W} = f \colon N\to M$. 
In other words,  we have two $B$-bordant normal 2-smoothings in the same fibration $ M \times BSPIN\to BSTOP$.
 We now apply Corollary \ref{cor:threethree} to complete the proof. \end{proof}

\section{Homotopy self-equvalences of $4$-manifolds }\label{sec:four}
We will  recall a braid diagram relating homotopy self-equivalences to bordism theory (see Hambleton and Kreck \cite{Hambleton:2004}). 
The proof of Theorem B will use an approach to cancellation introduced by Pamuk \cite{Pamuk:2009,Pamuk:2010} based on this braid.

Let $\hept M$ denote the group of
homotopy classes of homotopy self-equivalences, preserving both
the given orientation on $M$ and a 
fixed base-point $x_0 \in M$. There are also
 ``pointed" versions of 
the space $\heqpt{B}$ of base-point preserving homotopy
equivalences of $B$ (the algebraic $2$-type of $M$).
The main result  of \cite{Hambleton:2004} for spin manifolds is expressed in a commutative braid of interlocking exact sequences:
\vskip .3cm
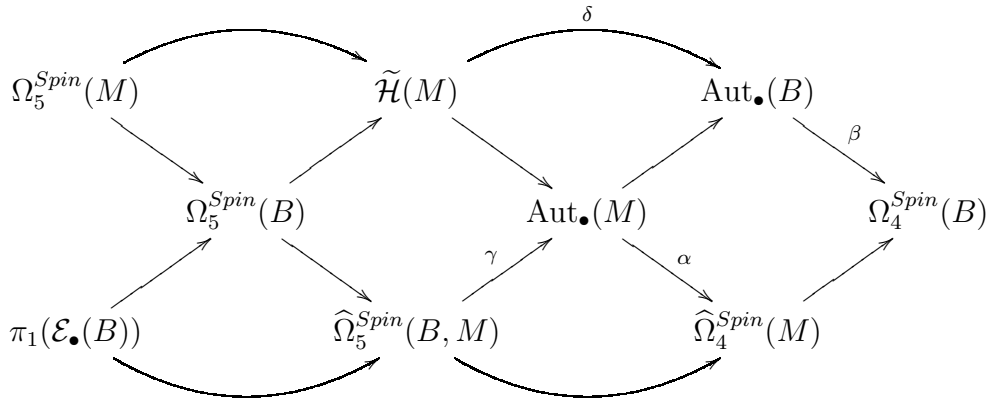
$$\begin{matrix}
\xymatrix@!C@C-30pt{
\Ospin_5(M)    \ar[dr] \ar@/^2pc/[rr]   &&
\htildeM  \ar[dr] \ar@/^2pc/[rr]^{\delta} &&
\hept{B}\ar[dr]^{\beta}  \\
& \Ospin_5(B) \ar[dr] \ar[ur]  &&
\hept{M}  \ar[dr]^{\alpha}\ar[ur] &&\Ospin_4(B) \\
 \pi_1(\heqpt{B}) \ar[ur] \ar@/_2pc/[rr] && \rOspin_5(B,M)
  \ar[ur]^\gamma \ar@/_2pc/[rr]&&
\rOspin_4(M)\ar[ur]
}\end{matrix}
$$
\vskip .8cm\noindent
valid
for any closed, oriented smooth or topological spin $4$-manifold $M$ (see \cite[Theorem 2.16]{Hambleton:2004}).
The maps labelled $\alpha$ and  $\beta$  are
not necessarily group homomorphisms, so exactness  is understood in the
sense of  ``pointed sets" (meaning that $image = kernel$, where $kernel$ is the pre-image of the base point).

\medskip
Here is an informal description of the other objects in the braid. 
\begin{enumerate}
\item The group $\hM$ consists of 
oriented $h$-cobordisms $W^5$ from $M$ to $M$, under the
equivalence relation induced by $h$-cobordism relative to the
boundary. The orientation of $W$ induces  opposite
orientations on the two boundary components $M$.
An $h$-cobordism gives a homotopy self-equivalence of $M$,
and we get a homomorphism $\hM\to \he{M}$.

\item The natural map
$c\colon M \to B$ is $3$-connected, and we refer
to this as the classifying map of $M$. There is an
induced homomorphism $\he M \to \he B$,
the group of homotopy
classes of homotopy self-equivalences of $B$, by 
obstruction theory and the
naturality of the construction.

\item If $M$ is a spin manifold, 
we use  the  smooth (or topological) 
bordism groups $\Ospin_n(B)$. 
By imposing the requirement that the reference maps
to $M$ must have degree zero, we obtain modified bordism
groups $\rOspin_4(M)$ and $\rOspin_5(B,M)$.

\item The map $\alpha\colon \hept{M} \to \rOspin_4(M)$ is given by $\alpha(f) = [M, f] - [M,\id]$, and the map $\beta\colon \hept{B} \to \Ospin_4(B)$ is given by $\beta(\phi) = [M, \phi\circ c] - [M, c]$. For the map $\gamma$, see \cite[\S 2.5]{Hambleton:2004}.

\item A variation of $\hM$, denoted $\htildeM$, will
also be useful. This is the group of oriented
bordisms $(W,\bd_{-}W, \bd_{+} W)$ with
$\bd_{\pm}W = M$, equipped with a map
$F\colon W \to M$.
We require  the restrictions $F|_{\bd_{\pm}W}$
 to the boundary components
to be   homotopy equivalences (and the identity on
 the component $\bd_{-} W$). The equivalence
relation on these objects is induced by
bordism (extending the map to $M$) relative to the boundary
(see \cite[Section 2.2]{Hambleton:2004}  for the details).
\end{enumerate}

\section{The image of the fundamental class}\label{sec:fourB}
Let $B:= B(M)$ denote the algebraic $2$-type of a closed oriented topological $4$-manifold  $M$with infinite fundamental group $\pi$.  
We will indicate the places where we assume that $\pi$  has one end, or equivalently that $H^1(\pi; \La) = 0$. By Poincar\'e duality, this implies that $H_3(\wM;\bZ) = H_3(M;\La) = 0$. Since $\pi_3(B) = 0$, we also have $H_3(\wB;\bZ) = 0$.

\begin{remark}\label{rem:sixzero}  For some applications we will assume that the end homology  $H_1^e(E\pi) = 0$. This imposes restrictions on the low-dimenaional cohomology of $\pi$, via the exact sequence
$$0 \to \Ext_\bZ(H^2_e(E\pi),  \bZ) \to H_1^e(E\pi; \bZ) \to \Hom_\bZ(H^1_e(E\pi), \bZ),\to 0$$
from \cite[Proposition 2.9]{Laitinen:1996}, and the isomorphisms $H^q_e(E\pi) \cong H^{q+1}(\pi;\La)$, for $q >0$. 
Therefore if $H^e_1(E\pi;\bZ) = 0$ then $H^2(\pi;\La)$ is all torsion and $H^3(\pi; \La$ is torsion-free.
For example, $H_1^e(E\pi) = 0$ whenever $H^q(\pi;\La) = 0$ for $q = 2,3$.

The statement that $H^2(\pi;\La)$ is free abelian for all finitely-presented groups (which would imply that $\pi_2(M)$ is free abelian) is said to be a conjecture of Hopf \cite[Remark 4.5]{Geoghegan:1985}). The conjecture is still open, although it  has been verified in some cases (see \cite{Mihalik:1992,Mihalik:1992a}).
\end{remark}

Our most general result so far about the image of fundamental class requires some group cohomology conditions (introduced in  \cite[Definition 3.1]{Hambleton:2019}). In the setting of Theorem B, these conditions are satisfied.
\begin{definition}\label{def:tamecoh}
A finitely presented group $\pi$  has \emph{tame cohomology} if the following conditions hold:
\begin{enumerate}
\item $ \Hom_\La(H^2(\pi;\La), \La) = 0$
\item  $\Hom_\La(H^3(\pi;\La), \La) = 0$
\item  $\Ext^1_\La(H^3(\pi;\La), \La) = 0$.
\end{enumerate}
\end{definition}

In applications of the braid diagram, it is important to understand the maps in the exact sequence
$$ \Ospin_5(B) \to \htildeM  \xrightarrow{\delta}
\hept{B} \xrightarrow{\beta}
 \Ospin_4(B). $$
 In particular, if $\phi\colon B \to B$ is a homotopy self-equivalence, we need to understand  the image $\phi_*(c_*[M]) \in H_4(B;\bZ)$ of the fundamental class $[M] \in H_4(M;\bZ)$ in order to compute $\beta(\phi) = [M, \phi\circ c] - [M,c] \in  \Ospin_4(B)$.

 We first need some information about $H_4(B;\bZ)$. Recall that we have an expression $H_4(\wB;\bZ) \cong \Gamma(\pi_2(B))$, in terms of Whitehead's $\Gamma$-functor (see \cite[Chap.~II]{Whitehead:1950}). In addition, we have the orientation class  $\mu_M \in H_4(B, \wB) = \bZ$, given in Definition \ref{def:enriched} as the image of the fundamental class $[M] \in H_4(M)$ under the composition $H_4(M) \to H_4(M, \wM) \xrightarrow{c_*}  H_4(B, \wB)$.
 
 \begin{proposition}\label{prop:sixone} Suppose that $\pi$ has one end.
 \begin{enumerate}
 \setlength{\itemsep}{0pt plus 5pt}
 \item The map $c_*\colon H_4(M;\bZ) \to H_4(B; \bZ)$ is injective.
 \item The composition
$\omega\colon H_4(M;\bZ) \xrightarrow{c_*} H_4(B;\bZ)  \xrightarrow{\cap\ }  \Hom_{\La}(H^2(B;\bZ), H_2(B;\bZ))$ 
induces the ordinary intersection form $q_M$.
\item If $\phi \in  \hept{B}$ is orientation-preserving, so that $\phi_*(\mu_M) = \mu_M \in H_4(B, \wB)$,  then $c_*[M] - \phi_*(c_*[M]) \in \Image(H_4(\wB;\bZ)\otimes_\La \bZ \to H_4(B;\bZ))$. 
\item If $\phi \in  \hept{B}$ induces the identity on $\twotypeM$, then $\phi_*(\mu_M) = \mu_M \in H_4(B, \wB)$.
 \end{enumerate}
 \end{proposition}
 \begin{proof} Here we will use homology with integer coeffiicients unless otherwise stated. For part (i) we compare that spectral sequences of the coverings $\wM \to M$ and $\wB \to B$, and note that $H_3(\wM) = H_4(\wM) = 0$ under our assumptions. The terms $E^2_{p,q} = \Tor^\La_p(\bZ, H_q(\wM))$ are mapped isomorphically for $q \leq 3$. 
 We have a commutative diagram:
 $$\xymatrix{  H_5(\pi;\bZ) \ar[r]^(0.4){ d_{5.0}^3}\ar@{=}[d]&\Tor^\La_2(\bZ, H_2(\wM)) \ar[r]\ar@{=}[d]&H_4(M, \wM) \ar[d]\ar[r]& H_4(\pi;\bZ) \ar@{=}[d]\\
  H_5(\pi;\bZ) \ar[r]^(0.4){ d_{5.0}^3}&\Tor^\La_2(\bZ, H_2(\wB)) \ar[r]& H_4(B, \wB) \ar[r] &H_4(\pi;\bZ)
 }$$
 We see that $\bZ = H_4(M) \cong H_4(M, \wM) \cong H_4(B, \wB)$ under the natural maps, and part (i) follows. By definition, $[M] \mapsto \mu_M$ under this composite isomorphism.
 
 For any $a, b \in H^2(M)$,  we have $x = a \cap [M]$ and $y = b\cap [M]$ in $H_2(M)$ under Poincar\'e duality. 
 Then $q_M(x,y) = \la a \cup b, [M]\ra = \la a \cap [M], b\ra$. Since $c$ is a $3$-equivalence, 
 $$q_M(x,y) = \la c^*\bar a \cup c^* \bar b, [M]\ra =  \la \bar a \cup  \bar  b,c_* [M]\ra$$
 for some $\bar a, \bar b \in H^2(B)$. Therefore 
 $q_M(x,y) =  \la \omega ([M])(a), b\ra$. This gives part (ii).
 
 For part (iii),  
 we have  $c_*[M] - \phi_*(c_*[M]) \in \Image(H_4(\wB;\bZ)\otimes_\La \bZ \to H_4(B;\bZ))$, since 
$\mu_M$ generates $H_4(B, \wB)$, and  $\phi_*(\mu_M) = \mu_M$ by assumption.

For part (iv), we consider the exact sequence:
$$ \dots \to H_5(\pi;\bZ) \xrightarrow{\ d_{5.0}^3\ }\Tor^\La_2(\bZ, H_2(\wB)) \to H_4(B, \wB) \to H_4(\pi;\bZ) \to \dots $$
obtained from the spectral sequence of the covering $\wB \to B$. Since $H_4(B, \wB)  \cong \bZ$, and $\phi$ acts trivially on $\pi_1(M)$ and $\pi_2(M)$, the result follows.
 \end{proof}

We recall from Definition \ref{def:equiv} that the class $\trf [M] \in  H_4^\LF(\widetilde M; \bZ) \cong H_4(M;\hL)$ induces the equivariant intersection form $s_M$ on $\pi_2(M)$. In this expression, 
$$\hL = \{ \sum n_g \cdot g \vv {\rm \ for\ } g \in G, {\rm \ and\ } n_g \in \bZ\}$$
denotes the formal (possibly infinite) integer linear sums of group elements (see Section \ref{sec:seven}). The transfer map can be expressed as the change of coefficients homomorphism 
$\trf \colon H_4(M;\bZ) \to H_4(M; \hL)$ via the map $1 \mapsto \hN := \sum \{ g \vv g \in \pi\}$.  The image of the transfer map therefore lands in 
the $\pi$-fixed subgroup $H_4^\LF(\widetilde M; \bZ)^\pi$. 

We now translate this information to $B$. The transfer map
$$\trf\colon H_4(B;\bZ) \to H_4(B; \hL)^\pi \cong H_4^\LF(\widetilde B; \bZ)^\pi$$
is similarly defined by the coefficient inclusion $\bZ \subset \hL$ and the identification provided by Corollary \ref{cor:sevenzero}.
Define a map 
$$\omega\colon H_4^\LF(\widetilde B; \bZ)^\pi \to \Hom_{\La}(H^2(B;\La), H_2(B;\La))$$
by setting $\omega(z) = z \cap c$, for  $z \in H_4^\LF(\widetilde B; \bZ)^\pi$ and $c \in H^2(B;\La)$.

If $\alpha \in \Image(H_4(\wB;\bZ)\otimes_\La \bZ \to H_4(B;\bZ))$, we can pick a lift 
$\hat\alpha \in H_4(\wB;\bZ)$, and   then the restriction of the transfer map
$\trf (\alpha) \in H_4^\LF(\widetilde B; \bZ)^\pi$ is just the image of 
 $\hat\alpha\otimes_\La \hN  \in H_4(\wB;\bZ)\otimes_\La \hL$. 
 This expression is independent of the choice of lift $\hat \alpha \mapsto \alpha$, since elements of the form $(1-g)\otimes_\La \hN = 0$ for all $g \in\pi$.

\begin{definition}\label{def:sixthree} A $\La$-module $L$ is called \emph{torsionless} if there exists an $\La$-embedding $L \subset F$, where $F$ is a finitely generated free $\La$-module.
 The module $L$ is called \emph{strongly torsionless} if additionally the induced map $\Gamma(L) \otimes_\La \bZ \to \Gamma(F) \otimes_\La \bZ$
 is injective. 
\end{definition}
We remark that these properties depend only on the stable class of the module.
Note that if $L = H_2(\wB) \cong \pi_2(M)$, we have $\Gamma(L) = H_4(\wB)$. For the terminology see \cite[\S 4.4, pp.~476-477]{Bass:1960} and the statement that the dual of a finitely generated $\La$-module embeds in a finitely generated free module.

\begin{lemma}\label{lem:fourBone} 
Assume that $\pi$ has one end. Then
\begin{enumerate}
\item  The  image $\omega(\trf (c_*[M]))$ induces the equivariant intersection form $s_M$.
\item The natural map $H_4(\wB) \otimes_\La \hL \to H_4(B;\hL)\cong H_4^\LF(\widetilde B)$ is injective.
\item If $H_2(\wB)$ is strongly torsionless, and $\pi$ has tame cohomology,  then the composite 
$$H_4(\wB)\otimes_\La \bZ \xrightarrow{\ \trf\ }   H_4^\LF(\widetilde B)^\pi \xrightarrow{\ \omega\ } \Hom_{\La}(H^2(B;\La), H_2(B;\La))$$
 is injective.
\end{enumerate}
\end{lemma}
\begin{proof} 
The first statement follows from the definition of $s_M$. Since $c\colon M \to B$ is a $3$-equivalence, the cap product
$$\cap\, \trf(c_*[M])\colon H^2(B;\La) \to H_2(B;\La)$$
is an  isomorphism by Poincar\'e duality. The hermitian form induced by the cup product
$$ H^2(B;\La) \times H^2(B;\La) \to H^4(B;\La) \to H^4(M;\La)\cong \bZ$$
may be identified with $s_M$ (see Definition \ref{def:equiv}). 
For part (ii) we compare the spectral sequences under the map  $H_4(M;\hL) \to H_4(B;\hL)$, starting with
$$E^2_{p,q}(M) = \Tor^{\La}_p(H_q(\wM), \hL) \to E^2_{p,q}(B) = \Tor^{\La}_p(H_q(\wB), \hL).$$
Note that $H_3(\wM) = H^1(\pi;\La) = 0$, by our assumption that $\pi$ has one end.
Since $H_k(M;\hL) = 0$ for $k \geq 5$ and $H_4(\wM) = 0$, the differential $d_3^{5,0}$ is injective, and the differential 
$d_3^{6,0}$ is surjective (in the spectral sequence for $H_4(M;\hL)$). By comparison, there are no non-zero differemtials hitting the $(0,4)$ position in the spectral sequence for $H_4(B;\hL)$. Hence the term $E^2_{0,4}(B) =  H_4(\wB) \otimes_\La \hL$ survives, and injects into $H_4(B;\hL)$.

For part (iii): since
  $L = H_2(\wB) = H_2(B;\La)$ is  torsionless there exists an $\La$-embedding $e \colon L \subset F$, where $F$ is a finitely generated free $\La$-module.
Let $P$ denote the 2-stage Postnikov tower with $\pi_1(P) = \pi$, $\pi_2(P) = F$, and $k$-invariant pushed forward by the induced map  $H^3(\pi; \pi_2(B)) \xrightarrow{\ e_*\ } H^3(\pi; \pi_2(P))$. We have a commutative diagram
\eqncount
\begin{equation}\label{diag:sixfour}
\vcenter{
\xymatrix@R-5pt{
H_4(\wB)\otimes_\La \bZ\   \ar[r]^(0.5){\trf}  \ar[d]^{e_*}&H_4^\LF(\widetilde B)^\pi \ar[r]^(0.3){\omega_B}\ar[d]^{e_*}&
\Hom_{\La}(H^2(B;\La), H_2(B;\La))\ar[d]^{\Hom(e^*, \, e_*)}\\
H_4(\wP)\otimes_\La \bZ\   \ar[r]^(0.5){\trf} &H_4^\LF(\widetilde P)^\pi \ar[r]^(0.3){\omega_P}&
\Hom_{\La}(H^2(P;\La), H_2(P;\La))
}}
\end{equation}
The left-hand vertical arrow is injective since $H_4(\wP) = \Gamma(F)$ and we have assumed that $L$ is strongly torsionless. We also need some more information about the sequence
$$0 \to H^2(\pi;\La) \to H^2(P;\La) \to \Hom_\la(\pi_2(P), \La) \to H^3(\pi;\La) \to 0.$$
Under  the tame cohomology assumption (ii)   of Definition \ref{def:tamecoh}, we have an injection:
$$ 0 \to \Hom_\La(\Hom_\La(\pi_2(P), \La), \pi_2(P)) \to \Hom_\La(H^2(P;\La), \pi_2(P))$$
after applying  $\Hom_\La(-, \pi_2(P))$ to each term, since $\pi_2(P) = H_2(P;\La)$ is free over $\La$.
If  we add conditions (i) and (iii),   then we get an isomorphism
$$ \Hom_\La(\Hom_\La(\pi_2(P), \La), \pi_2(P)) \cong\Hom_\La(H^2(P;\La), \pi_2(P)).$$
As a consequence, we can use the identification $\omega_P\colon H_4^\LF(\widetilde P)^\pi  \to \Hom_\La(\Hom_\La(F, \La), F)$ in studying  the diagram \eqref{diag:sixfour}.

To show that the lower horizontal composite $\omega_P\circ \trf$ is injective, we 
recall the proof of \cite[Lemma 5.15]{HKT09}. If $F = \La^r$, we have a $\bZ$-base $\{a_i\}$ for $F$ consisting of elements $a_i = ge_j$, for some $g \in \pi$, where $\{e_j\}$ denotes a $\La$-base for $F$. Following \cite[p.~63]{Whitehead:1950}, define
$$F^* = \{ \phi\colon F \to \bZ \vv \phi(a_i) = 0 \text{\ for almost all\ } i\}.$$
Let $\{a_i^*\}$ denote the dual basis for $F^*$. We say that a homomorphism $f\colon F^* \to F$ is \emph{admissible} of $f(a_i^*) = 0$ for almost all $i$, and that $f$ is \emph{symmetric} if $a^*fb^* = b^*fa^*$ for all $a^*, b^* \in F^*$. Then
$$\Gamma(F) \cong \{f \colon F^* \to F \vv f \text{\ is symmetric and admissible}\}.$$
We now observe that $\Hom_\La(F, \La) \cong F^*$, and we have a commutative diagram:
$$\xymatrix@R+3pt{& H_4^{\LF}(\wP;\bZ)^\pi \ar[r]^(0.4){\omega}& \Hom_\bZ(F^*, F)^\pi\\
\Gamma(F)_\pi \ar[ur]^{\trf}\ar@{=}[r]&H_4(\wP;\bZ)_\pi\  \ar[u]^N \ar@{>->}[r]^(0.4){}&\Hom^a_\bZ(F^*, F)_\pi \ar[u]^N
}$$
where $\Hom^a$ denotes the admissible homomorphisms, and the norm maps $N\colon L_\pi \to L^\pi$ are formally defined for any $\La$-module by applying the operator $\hN = \sum\{ g\vv g\in \pi\}$. Here $L_\pi = L \otimes_\La \bZ$ is the co-fixed set, and $L^\pi$ is the fixed set. For the middle term, the norm map $N$ is induced by the coefficient map $H_4(P;\bZ) \to H_4(P; \hL) \cong 
H_4^{\LF}(\wP;\bZ)$ sending $1 \to \hN \in \hL$. The right-hand norm map in the diagram is the direct sum of the norm maps
$$ N \colon \Hom^a_\bZ(\La^*, \La)_\pi \to \Hom_\bZ(\La^*, \La)^\pi$$
and the rest of the argument to show that $N$ is injective is explained in detail on \cite[p.~144]{HKT09} (note that the reference to Whitehead \cite{Whitehead:1950} has been corrected here). To check that the map 
$$H_4(\wP;\bZ)_\pi\ \to \Hom^a_\bZ(F^*, F)_\pi $$
is injective  (but not bijective), it is convenient to use the description for $\Gamma(\La)$ given in \cite[Lemma 2.2]{Hambleton:1988a}.
Hence $\omega_P \circ \trf$ is injective, and from  diagram \eqref{diag:sixfour} we conclude that $\omega_B \circ \trf$ is injective as required.
\end{proof}

\begin{corollary}\label{cor:sixtwo}  Suppose that $\pi$ has one end, and $\pi$ has tame cohomology. 
 If  $H_2(\wB)$ is strongly torsionless, and $\phi\in  \hept{B}$ induces an oriented isometry of the quadratic $2$-type
 $\fQ(M)$,  then 
$\phi_*(c_*[M]) = c_*[M] \in  H_4(B;\bZ)$.
\end{corollary}
\begin{proof} By Proposition \ref{prop:sixone} (iii),  we have
$$\alpha:= c_*[M] - \phi_*(c_*[M]) \in \Image(H_4(\wB;\bZ)\otimes_\La \bZ \to H_4(B;\bZ)).$$
since $\phi$ is oriented. Since $\phi$ is an isometry of the quadratic $2$-type, 
Lemma \ref{lem:fourBone}(i) gives $\omega(\trf (c_*[M])) = \omega(\trf (\phi_*(c_*[M])))$, and Lemma \ref{lem:fourBone}(iii) implies that $c_*[M] = \phi_*(c_*[M])$.
\end{proof}

\section{Applications}\label{sec:fourA}

In this section we will describe a general process for establishing results like Theorem B.

\begin{theorem}\label{thm:main} Let $\pi$ be a discrete group of type $F$  with one end,   satisfying properties \textup{(W-AA)}.  Let
$M$ and $N$ be closed, smooth (topological), spin $4$-manifolds  with  fundamental group $\pi$,  and  isometric oriented  quadratic $2$-types.   
 If $M$ and $N$ are stably diffeomorphic (homeomorphic) and the  composite 
  $$H_4(\wB)\otimes_\La \bZ \xrightarrow{\ \trf \ }   H_4^\LF(\widetilde B)^\pi \xrightarrow{\ \omega\ } \Hom_{\La}(H^2(B;\La), H_2(B;\La))$$
 is injective, then $M\# r(S^2 \times S^2)$ is $s$-cobordant to $N\# r(S^2 \times S^2)$,
  provided that $r \geq \sr(\pi)$. 
\end{theorem}

We will discuss the topological case, and note that the arguments in the smooth case follow the same steps.
If $M$ and $N$ are stably homeomorphic, then we can construct  a $5$-dimensional spin  bordism   $(V; M, N)$ between $M$ and $N$ over $K(\pi,1)$ (with respect to compatible spin structures). By \cite[Chapter 9]{Freedman:1990}, there is a topological handlebody structure on $V $ relative to its boundary. 

As in the proof of the $h$-cobordism theorem, we may assume that $V$ consists of $2$-handles and $3$-handles, so that at a middle level $V_{1/2} \approx M \# t (S^2 \times S^2) \approx N\# t(S^2 \times S^2)$, for some $t \geq 0$.

\begin{proof} Here are the remaining steps in the proof.
\begin{enumerate}
\setlength{\itemsep}{0pt plus 10pt}
\item\label{itm:one}  Let  $\theta\colon \quadtypeMb \xrightarrow{\cong} \quadtypeNb$ be an orientation-preserving isometry of the  quadratic $2$-types of $M$ and $N$, and use it together with a given isomorphism of their fundamental groups to identify their algebraic $2$-types $B := B(M) = B(N)$.

\item\label{itm:two} 
Let $h\colon N \# t (S^2 \times S^2) \to M\# t(S^2 \times S^2)$ be a stable orientation-preserving homeomorphism, with
$$h_*\colon \pi_2(N) \oplus H(\La^t) \xrightarrow{\cong}  \pi_2(M) \oplus H(\La^t) $$
the induced isometry of their stabilized equivariant intersection forms. We may assume that the $k$-invariants are preserved.  
Let $M_t := M \# t(S^2 \times S^2) $ and let $B_t$ denote the stabilized algebraic $2$-type for $ M \# t(S^2 \times S^2) $ and
$ N \# t(S^2 \times S^2) $.

\item\label{itm:three} Let $\gamma := h_* \circ (\theta \oplus \id_t)$ be the induced oriented self-isometry of the stabilized quadratic $2$-type of $M_t$, where $\id_t \colon H(\La^t)\to H(\La^t)$ denotes the identity map on the added hyperbolic summand. Then there exists a homotopy self-equivalence $\phi \colon B_t \to B_t$ such that $c_*^{-1}\circ \phi_*\circ c_* = \gamma$.

\item\label{itm:four}  By Proposition \ref{prop:sixone}(iii), we have
$$\alpha:=\phi_*(c_*[M_t]) - c_*[M_t] \in \Image\{H_4(\wB_t;\bZ) \to H_4(B_t;\bZ)\}, $$
since $\phi$ induces an oriented isometry of the quadratic $2$-type. Moreover, since the composite $\omega\,\circ \,\trf $ is injective (by assumption),  it follows from Lemma \ref{lem:fourBone}(i)  that  $\phi_*(c_*[M_t]) = c_*[M_t] \in H_4(B_t;\bZ)$.

\item\label{itm:foura} By \cite[Theorem 1.1]{Hambleton:2018}, there exists a homotopy self-equivalence $g\colon M_t\to M_t$ such that $c \circ g \simeq \phi \circ c$. Since $\kappa_2\colon H_2(\pi;\cy 2) \to L_4(\Zpi)$  is injective (by condition (W-AA)), the normal invariant $\eta(g) \in H_2(M_t;\cy2)$ lies in $\ker\{H_2(M_t;\cy2) \to H_2(\pi;\cy2)\}$. 
By \cite[Theorem 19]{Kirby:2001}, after composing $g$ with suitable  self-equivalences given by pinch maps inducing the identity on $\pi_2(M_t)$, we may assume that the normal invariant $\eta(g) \in H_2(M_t;\cy2)$ vanishes. Therefore $(M, g)$ is normally cobordant to $(M,\id)$ and we have 
$$\alpha(g) = 
 [M_t, f] - [M_t,\id] = 0 \in \rOspin_4(M_t).$$

\item\label{itm:seven} We use the braid for the stabilized $M_t$ and its $2$-type $B_t$ to show that $[\phi]$ is the image of an element $[(W, F) ]\in \htildeMt$ under the map $\delta\colon \htildeMt \to \hept{B_t}$. The image of $[(W,F)]$ in $\hept{M_t}$ in the braid is represented by the self-equivalence $g:= F|_{\bd_{+} W}\colon M_t \to M_t$. Note that  $[g] \mapsto [\phi] \in \hept{B_t}$ under the map $\hept{M_t}\to \hept{B_t}$ in the braid, so that $ g_* = h_* \circ (\theta \oplus \id_t)$.

\item\label{itm:eight}  There is an exact sequence:
$$L_6(\Zpi) \to  \hMt \to \htildeMt \to L_5(\Zpi)$$
and the map $ \htildeMt \to L_5(\Zpi)$ is given by the (modified) surgery obstruction of the map $F\colon W \to M_t\times I$, relative to the boundaries (see \cite[p.~163]{Hambleton:2004}).

\item\label{itm:nine}We now apply Corollary \ref{cor:threethree} to $(W, F)$, regarded as a bordism from the normal $2$-smoothing  $\id\colon M_t \to M_t$ to itself, over the reference map $F\colon W \to M$. For any given $r \geq \sr(\pi)$, we can realize an element $[\alpha_r] = -\sigma(F) \in L_5(\Zpi)$, with  $\alpha_r \in SU_r(\La)$, by a stabilized normal cobordism,  and attach it to $(W,F)$ along $M_t \# r(S^2\times S^2)= \bd_{+} W  \# r(S^2\times S^2)$   (see the proof of \cite[Theorem 3.1]{Cappell:1971}). 

The resulting cobordism has zero surgery obstruction, so after  performing surgery (relative to the boundary),  the result is an $s$-cobordism $(W', F')$ of $M_t \# r(S^2\times S^2)$. By construction, $F' |_{\bd_{-} W} = \id_{M_t \# r(S^2\times S^2)}$ and 
$$F' |_{\bd_{+} W} = f\circ g\colon M_t \# r(S^2\times S^2) \to M_t \# r(S^2\times S^2),
$$
where $(M_t \# r(S^2\times S^2), f)$ is a (simple) homotopy self-equivalence, such that $f_*$ induces the identity on $\pi_2(M_t)$.

 \item\label{itm:ten} We now return to decompose the spin bordism between $M$ and $N$ as follows:
 $$ V =  M\times [0,1/4] \cup \text{\{2-handles\}} \cup \text{\{3-handles\}} \cup N \times [3/4, 1]$$
 As above, let   $V_{1/2} $ denote a middle level containing no critical points, so that the $2$-handles are all attached below $V_{1/2} $, and the $3$-handles attached all above $V_{1/2}  $. 
 
 Let 
 $V = V[0,1/2] \cup V[1/2,1]$ denote the lower and upper parts of $V$, joined along their common boundary $V({1/2})$ by the stable homeomorphism 
 $$h \colon M \# t (S^2 \times S^2) \to  N\# t(S^2 \times S^2)$$
  used in the steps above.  We then stabilize $V$ to $V'$ by connected sum with $r(S^2 \times S^2 \times [0,1])$ along  small disjoint embeddings of $D^4 \times [0,1] \subset V$, so that $\bd_{-} V' =  M_r:= M \# r(S^2 \times S^2)$ and $\bd_{+} V' =  N_r:= N \# r(S^2 \times S^2)$. We now have the stabilized decomposition
 $$V' =  V'[0,1/2] \cup V'[1/2,1],$$
 where 
 $\bd_{+} V'[0,1/2] = M_t \# r(S^2 \times S^2)$ and $\bd_{-} V'[1/2,1] = N_t  \# r(S^2 \times S^2) $.  The final step is to  glue  the $s$-cobordism $(W',F')$ in between the two halves to produce
 $V'' = V'[0,1/2] \cup W' \cup  V'[1/2,1]$, with $\bd_{\pm}  V'' = \bd_{\pm} V'$.
 
 \item\label{itm:eleven} We claim that $V''$ is an $s$-cobordism from $M_r$ to $N_r$. To see this, we check that the new attaching maps of the $3$-handles cancel the ascending $2$-handles. To keep track of the induced maps, let $L _M= \pi_2(M)$, 
 $L _N= \pi_2(N)$, $H_t = H(\La^t)$ and $H_r = H(\La^r)$. Then
 $$\pi_2(M_t \# r(S^2 \times S^2)) = L_M\oplus H_t \oplus H_r.$$
  The bordisms $V'[0,1/2]\cup W' $ and $ V'[1/2,1]$ are glued together along $\bd_{-} V'[1/2,1]$ by attaching the $3$-handles. The attaching maps are algebraically determined by the induced map on homology;
 $$ (h^{-1}_* \oplus  \id_r) \circ f_* \circ (g_* \oplus \id_r) \colon L_M \oplus H_t \oplus H_r \to L_N \oplus H_t \oplus H_r.$$
 Since $f_*$ induces the identity on $\pi_2(M_t)= L_M \oplus H_t$, and $h_*^{-1} \circ g_* = \theta \oplus \id_t$. we have
 $$  \left ((h^{-1}_* \oplus  \id_r) \circ f_* \circ (g_* \oplus \id_r)\right )(u, v, 0) = (\theta (u), v, 0),$$
 for all $ (u, v, 0) \in L_M \oplus H_t \oplus H_r $. 
 
 This formula shows that the $3$-handles from   $V'[1/2,1]$ (the upper half ) algebraically cancel  the $2$-handles  from  $V'[0,1/2]$ (the lower half), and these together give a standard hyperbolic base for  the summand $H_t$. Hence $V''$ is an $s$-cobordism between $M_r$ and $N_r$, and the proof of Theorem \ref{thm:main} is complete.
\end{enumerate}
\end{proof}

\begin{proof}[The proof of Theorem B] If $\pi = \pi_1(M)$ is the fundamental group of a closed, oriented aspherical $3$-manifold, then the Farrell-Jones conjectures hold for  $\pi$  (see \cite[Corollary 1.3]{Bartels:2014}) and $\pi$ has the properties (W-AA).
 Moreover, 
 $\gdim(\pi) = 3$, $H^1(\pi;\La) = 0$  and $H^3(\pi;\La) = \bZ$.
 
  By \cite[Lemma 6.1]{Hambleton:2019} we know that $\pi_2(M)^*$ is a stably free $\La$-module, and since $H^2(\pi;\La) =0$, we have a short exact sequence
$$ 0  \to H^2(M;\La) \to \Hom_\La(\pi_2(M), \La) \to H^3(\pi;\La) \to 0$$
which is isomorphic (by Shanuel's Lemma) to
$$0 \to I(\pi) \oplus F_0 \to \La \oplus F_0 \to \bZ \to 0,$$
 after stabilization if necessary, where $I(\pi)$ denotes the augmentation ideal of $\Zpi$,  $F_0$ is a (stably) free, finitely generated $\La$-module, and $L:=   I(\pi) \oplus F_0 $ is a stabilization of $\pi_2(M) \cong H^2(M;\La)$.
 Let $F = \La \oplus F_0 $ so that $L = I(\pi) \oplus F_0$ embeds in $F$ with quotient $\bZ$. In  particular, $\pi_2(M)$ is torsionless. 
 
By \cite[Proposition 4.1]{Hambleton:2019}, the fundamental group $\pi$ has tame cohomology.  It is now easy to verify the other conditions of Lemma \ref{lem:fourBone} needed to apply Theorem \ref{thm:main}.  
 
  In order to check that $\pi_2(M)$ is strongly torsionless, it is enough to show that the induced map
   $ \Gamma(L )\otimes_\La \bZ \to   \Gamma(F) \otimes_\La \bZ$ is injective, since this is a stable condition. From the additivity formula, we have a commutative diagram of $\La$-modules:
   $$\xymatrix{\Gamma(L) \ar[r]^(0.2)\cong \ar[d] & \Gamma(I(\pi)) \oplus \Gamma(F_0) \oplus I(\pi)\otimes_\bZ F_0 \ar[d]\\
   \Gamma(F) \ar[r]^(0.2)\cong & \Gamma(\La) \oplus \Gamma(F_0) \oplus \La\otimes_\bZ F_0
   }$$
   Since the additive decompositions are natural, we can consider the vertical maps separately. 
   By \cite[Lemma 2.3]{Hambleton:1988a}, there is a $\La$-isomorphism $ \Gamma(I(\pi)) \oplus \La \cong \Gamma(\La)$, so the first vertical map is split injective. The middle vertical maps is the identity, and the third vertical map is again a split injection over $\La$ since the sequence
   $$ 0 \to I(\pi)\otimes_\bZ F_0 \to \La\otimes_\bZ F_0 \to \bZ \otimes_\bZ F_0 \to 0$$
   is an exact sequence of free $\La$-modules. 
    Hence 
   the induced map $ \Gamma(L )\otimes_\La \bZ \to   \Gamma(F) \otimes_\La \bZ$ is injective, and $L$ is strongly torsionless.
   \end{proof}
\begin{example} The assumptions of Theorem \ref{thm:main} apply to stabilizations of aspherical $4$-manifolds, such as $M =  T^4 \Sharp r(S^2 \times S^2)$, but not to stabilizations of $M = T^2 \times S^2$. \end{example}

\section{The main results of \cite{Hambleton:2019}  corrected}\label{sec:five}

To correct the statements and proofs of Theorem A  and Theorem 11.2 in \cite{Hambleton:2019}, we use the new stable range conditions. 
For the main result below: we need to assume that  $\pi$ has type $F_3$ in addition to $\cd(\pi) \leq 3$. This amounts to assuming $\gdim(\pi) \leq 3$.

\begin{thma} Let $\pi$ be a \RAAG\ defined by a graph $\Gamma$ with no $4$-cliques.
 Suppose that $M$ and $N$ are closed,
  \spinp, topological $4$-manifolds with fundamental group $\pi$. Then any isometry between the quadratic $2$-types of $M$ and $N$ is stably realized by an $s$-cobordism between $M \Sharp r(S^2 \times S^2)$ and $N \Sharp r(S^2 \times S^2)$, whenever $r \geq \max\{b_3(\pi), 6\}$.
\end{thma} 

 This  is a consequence of the main result  \cite[Theorem 11.2]{Hambleton:2019}, with a corrected stability bound from applying Corollary \ref{cor:threethree} in the last step of the proof. If $\cd(\pi) \leq 2$, then no stabilization is needed for this result and the next (see \cite[Theorem C]{HKT09}).

\begin{thmd}
 Let $\pi$ be a discrete group  with $\gdim(\pi) \leq 3$
  satisfying the properties \textup{(W-AA)}.
 If $M$ and $N$ are closed, oriented, \spinp, topological $4$-manifolds with fundamental group $\pi$, then any isometry between the quadratic $2$-types of $M$ and $N$ is stably realized by an $s$-cobordism between $M \Sharp r(S^2 \times S^2)$ and $N \Sharp r(S^2 \times S^2)$, for $r \geq \max\{b_3(\pi), \sr(\pi)\}$.
\end{thmd}

\begin{remark} Note that we  obtain $s$-cobordisms after  connected sum with a  \emph{uniformly bounded} number of  copies of $S^2 \times S^2$, where  the bound depends only on the fundamental group. In contrast, ``stable classification" results might require an unbounded number of stabilizations as the manifolds $M$ and $N$ vary. \end{remark}

The stable classification result,  \cite[Theorem B]{Hambleton:2019}, is not affected:  two closed, oriented \spinp, topological $4$-manifolds with $\cd(\pi) \leq 3$ are stably homeomorphic if and only if  their equivariant intersection form are stably isometric. For the restricted class of \spinp manifolds, this extends the stable classification obtained in \cite{Kasprowski:2017} for fundamental groups of closed, oriente, aspherical $3$-manifolds  to more general fundamental groups.

\begin{remark} 
The proof of \cite[Lemma 5.15]{HKT09} implicitly assumes that $\Hom_\La(H^2(\pi; \La), \La) = 0$. This can be justified since a group $\pi$ with  $\cd (\pi) \leq 2$ has tame oohomology by \cite[Proposition 4.1 and Lemma 4.4]{Hambleton:2019}.  At present we do not know whether every discrete group $\pi$ (or even every \RAAG) with $\cd(\pi) =3$ has tame cohomology (see \cite[Remark 3.2]{Hambleton:2019} for an example with $\cd(\pi) = 4$).
\end{remark}

\section{Appendix: Locally finite and end homology}\label{sec:seven}
Let $X$ be a closed, oriented, topological $n$-manifold with $\pi_1(X) = G$ infinite. The universal covering $\wX$ is a non-compact $n$-manifold, and we have two versions of Poincar\'e duality expressed in the following diagram:
$$\xymatrix{ H_c^q(\wX;\bZ) \ar[r]\ar[d]^{D}_\cong & H^q(\wX;\bZ)\ar[d]^{D}_\cong\\
H_{n-q}(\wX;\bZ) \ar[r] & H^{{\LF}}_{n-q}(\wX;\bZ)
}$$
where the duality map is induced by cap product with the transfer 
$$\trf[X] \in H_n^{\LF}(\wX;\bZ)$$
of the fundamental class of $X$ into the locally finite homology of its universal covering. 
The first version is a special case of the general Poincar\'e duality theorem
$$\cap\, [X] \colon H^q(X; L) \to H_{n-q}(X; L)$$
valid for any $\Lambda:= \bZ G$-module $L$. If we take $L = \Lambda$, then 
$$ H^q(X; \Lambda) \cong H_c^q(\wX;\bZ) \quad {\rm and} \quad H_{n-q}(X; \Lambda) \cong H_{n-q}(\wX; \bZ).$$
To express the second version (which involves locally finite homology) in these terms, we define 
$$\hL = \{ \sum n_g \cdot g \vv {\rm \ for\ } g \in G, {\rm \ and\ } n_g \in \bZ\}$$
as the formal (possibly infinite) integer linear sums of group elements. Then $\La \subset \hL$ and $\hL$ is a $\La$-module by formal multiplication

$$\big (\sum_g  n_g g\big ) \big (\sum_h m_h h\big ) = \sum_x \big (\sum_g n_gm_{g^{-1}x}\big ) x$$
which is defined since the coefficients $\{n_g\}$ in $\La$ are only non-zero for finitely many group elements. Note that $\hL = \Hom_\bZ(\La, \bZ)$ is a \emph{co-induced module} (see \cite[p.~67]{Brown:1994}).

From the general Poincar\'e duality theorem we have
$$\cap\, [X] \colon H^q(X; \hL) \xrightarrow{\hphantom{x}\cong\hphantom{x}} H_{n-q}(X; \hL)$$
and we claim that this recovers the second version of non-compact duality  for $\wX$ given above.

\begin{proposition}\label{prop:sevenfive} For any right $\La$-module $L$, there is an isomorphism $\Hom_\bZ(L , \bZ) \cong \Hom_\La(L, \hL)$ of $\La$-modules, which is natural
with respect to $\La$-maps $L \to L'$.
\end{proposition}
\begin{proof} We define a map $u \colon \Hom_\bZ(L , \bZ) \to \Hom_\La(L, \hL)$ by the formula $f \mapsto \hat f$, where
$$\hat f(x) = \sum_g f(xg^{-1})g \in \hL$$
for any $f \in  \Hom_\bZ(L , \bZ) $. Then $\hat f(xh) = \hat f(x)h$, for all $h \in G$. We define a map  $v\colon \Hom_\La(L, \hL) \to \Hom_\bZ(L , \bZ) $ by
the formula $\varphi \mapsto \varepsilon_1\varphi$, where $\varepsilon \colon \hL \to \bZ$ is given by $\varepsilon_1(\sum n_g g) = n_1$. It is not difficult to check that $u$ and $v$ and inverse $\La$-maps, and provide the claimed natural isomorphism. 

We check that the maps $f \mapsto \hat f$ and $\varphi \mapsto \varepsilon_1\varphi$ are left $\La$-module maps. Define a left $\La$-action on
$\Hom_\bZ(L , \bZ)$ by the formula $(h\cdot f)(x) = f(xh)$, for all $h \in G$, and on $ \Hom_\La(L, \hL)$ by  $(h \cdot \varphi)(x) = h\varphi(x)$.
Then $\widehat{(h\cdot f)}= h \cdot \hat f$ and $\varepsilon_1 (h \cdot \varphi) = h \cdot (\varepsilon_1\varphi)$. Then $(h_1\cdot (h_2\cdot f))(x) = 
(h_2\cdot f))(xh_1) = f(xh_1h_2) = ((h_1h_2)\cdot f)(x)$, and similarly for $\varphi$.
\end{proof}

\begin{corollary}There is a natural isomorphism of $\La$-module chain complexes $C^*(\wX; \bZ) \cong C^*(X; \hL)$.
\end{corollary}
\begin{proof} We have a natural isomorphism $\Hom_\bZ(C_q(\wX), \bZ) \cong \Hom_\La(C_q(\wX), \hL)$, for $q \geq 0$, and the differentials are induced by the boundary maps $\partial_q \colon C_q(\wX) \to C_{q-1}(\wX)$. \end{proof}

\begin{corollary}\label{cor:sevenzero} There is a $\La$-module isomorphism $H^{{\LF}}_{q}(\wX;\bZ) \cong H_q(X; \hL)$, for $q \geq 0$.
\end{corollary}
\begin{proof} Since $H^q(\wX; \bZ) \cong H^q(X; \hL)$ as $\La$-modules, the result follows from Poincar\'e duality.\end{proof}

\begin{remark} The same expression holds for  any finite-dimensional  $CW$-complex $K$ and its universal covering $\widetilde K$, by considering the boundary of a high-dimensional thickening of $K$ is a Euclidean space for dimension $2\dim K + 2$.
\end{remark}

As shown in Laitinen \cite[\S 3]{Laitinen:1996}, the Poincar\' e duality theorems can be extended to include \emph{end homology}, which we can now express as $H_{q-1}^e(\wX;\bZ) \cong H_{q}(X; \hL/\La)$, for $q \geq 0$.

\begin{proposition}\label{prop:sevenone}There is a commutative diagram relating two long exact sequences by Poincar\'e duality:
\eqncount
\begin{equation}
\vcenter{\xymatrix{ \dots \ar[r] &H_c^q(\wX;\bZ) \ar[r]\ar[d]^{D}_\cong & H^q(\wX;\bZ)\ar[r]\ar[d]^{D}_\cong & H^q_e(\wX;\bZ)\ar[d]^{D}_\cong\ar[r]& \dots\\
\dots \ar[r] &H_{n-q}(\wX;\bZ) \ar[r] & H^{{\LF}}_{n-q}(\wX;\bZ) \ar[r] & H^{e}_{n-q-1}(\wX;\bZ)\ar[r] & \dots
}}\end{equation}
\end{proposition}
\begin{proof}
Poincar\'e duality gives $H^q(X; \hL/\La) \cong H_{n-q}(X; \hL/\La) \cong H_{n-q-1}^e(\wX;\bZ)$. The long exact  sequences are induced by the coefficient sequence 
$0 \to \La \to \hL \to \hL/\La \to 0$. In our setting $H^q_e(\wX;\bZ) \cong H^q(X; \hL/\La)$ and 
$H_{q}(X; \hL/\La)  \cong H_{q-1}^e(\wX;\bZ)$. 
\end{proof}

We conclude with some algebraic observations.

\begin{lemma}\label{lem:seventwo} Let $L$ be a $\La$-module which embeds in a projective $\La$-module. 
Then 
\begin{enumerate}
\item the map $L \otimes_\La \La \to L \otimes_\La \hL$ is injective;
\item $\Tor_k^\La(L, \hL) \to \Tor_k^\La(L, \hL/\La)$ is an isomorphism, for $k \geq 1$.
\item $\Hom_\La(\hL/\La, \hL) = 0$.
\item $\hL \otimes_\La \hL/\La =0$.
\end{enumerate}
\end{lemma}
\begin{proof} We may assume that $L \subset F$ for some free $\La$-module $F$. For any $ 0 \neq x_0 \in L$, there exists a $\La$-module map
$f\colon L \to \La$ with $f(x_0) \neq 0$. Recall that the universal property of tensor products is expressed in terms of \emph{balanced products}. If $R$ is a ring, $M$ is a right $R$-module, $N$ is a left $R$-module and $T$ is an abelian group, then a balanced product is a
bilinear map $b\colon M \times N \to T$ such that $b(m\cdot r, n) = b(m, r\cdot n)$, for all $m \in M$, $n \in N$ and $r \in R$.

Define
$b \colon L \times \hL \to \hL$ by $b(x, \hat \lambda) = f(x)\cdot \hat\lambda$, for all $x \in L$ and $\hat\lambda \in \hL$. Since $b$ is balanced over $\La$, and $b(x_0, 1) =  f(x_0) \cdot 1 \neq 0$, it follows that $x\otimes 1 \neq 0$.

For part (ii), we tensor the exact sequence $0 \to \La \to \hL \to \hL/\La \to 0$  with $L$ over $\La$, and consider the resulting long exact sequence.  Since $\Tor_k^\La(L, \La) = 0$ for $k \geq 1$, and  $\Tor_1^\La(L, \hL) \to \Tor_1^\La(L, \hL/\La)$ is surjective by part (i), the result follows.

For part (iii), use the sequence
$$0 \to \Hom_\La(\hL/\La, \hL)  \to \Hom_\La(\hL, \hL) \to \Hom_\La(\La, \hL)$$
where the second map is isomorphic to the injective map $ \Hom_\bZ(\hL, \bZ) \to \Hom_\bZ(\La, \bZ)$. In fact, since $\hL \cong \prod \bZ$ is a countable direct product (although uncountable as an abelian group), its $\bZ$-dual $ \Hom_\bZ(\hL, \bZ) \cong \bigoplus \bZ$ is the direct sum.

For part (iv), we use the bimodule structure on $\hL$. In general, if $R$ and $S$ are rings, $M$ is an $(R,S)$-bimodule, $N$ is a left $S$-module, and $T$ is a left $R$-module, then the universal property is expressed by $S$-balanced maps $b\colon M \times N \to T$, such that
$b(rm, n) = rb(m,n)$ and $b(ms, n) = b(m,sn)$.  Note that
 the right adjoint $\ad b\colon N \to \Hom_R(M, T)$ is a left $S$-module map. If $R=S$ we call $b$ an $R$-bilinear map.

We let $R=S=\La$, $M = \hL$, $N = \hL/\La$,  and claim that $\hL \otimes_\La \hL/\La =0$ if any such $R$-bilinear map
  $b\colon \hL \times \hL/\La  \to  \La$  with range $T = \La$ must be zero (this is an easy reduction). To verify this claim, suppose that $b$ is non-zero, then by composition with the inclusion $\La \subset \hL$,  
the right adjoint $ \ad \hat b \colon \hL/\La \to \Hom_\La(\hL,  \hL)$ is a non-zero $\La$-map. However, $\Hom_\La(\hL, \hL) \cong \Hom_\bZ (\hL, \bZ)\subseteq \Hom_\bZ(\La, \bZ) \cong  \hL$. Since $\Hom_\La(\hL/\La, \hL) = 0$ by part (iii), we have a contradiction and hence $b\equiv 0$.
\end{proof}

\begin{remark}\label{rem:seventhree} The module $L = \bZ$ does not embed in a free $\La$-module (unless $G$ is finite): a sufficient condtion is that $L = B^*$ for some finitely generated $\La$-module $B$ (see Bass \cite[p.~477]{Bass:1960}). Note that $\bZ\otimes_{\La}\hL =0$ (see \cite[\S 2.5, \S 4.3]{Strebel:1977}, or 
\cite[Ex.~4(c), p.~71]{Brown:1994}) so some condition on $L$ is needed for part (i). 
\end{remark}



\providecommand{\bysame}{\leavevmode\hbox to3em{\hrulefill}\thinspace}
\providecommand{\MR}{\relax\ifhmode\unskip\space\fi MR }
\providecommand{\MRhref}[2]{%
  \href{http://www.ams.org/mathscinet-getitem?mr=#1}{#2}
}
\providecommand{\href}[2]{#2}

\end{document}